\documentclass{amsart}
\usepackage{amsmath, mathtools}
\usepackage{amssymb}
\usepackage{tikz}
\usepackage{graphicx}
\usetikzlibrary{graphs}
\usepackage{amscd}
\usepackage[mathscr]{euscript}
 \let\mathscr\relax
\usepackage[scr]{rsfso}
\usepackage[all,cmtip]{xy}
\input xy
\xyoption{all}
\newtheorem{Lemma}{Lemma}[section]
\newtheorem{remark}[Lemma]{Remark}

\newtheorem{theorem}[Lemma]{Theorem}

\newtheorem{proposition}[Lemma]{Proposition}
\newtheorem{corollary}[Lemma]{Corollary}

\newcommand{\Cal}[1]{{\mathcal #1}}

\newcommand{\Hom}{\operatorname{Hom}}

\newcommand{\figref}[1]{\figurename~\ref{#1}}

\DeclareMathOperator{\Triv}{\mathbf {Triv}}
\DeclareMathOperator{\Equiv}{\mathbf {Equiv}}
\DeclareMathOperator{\ParOrd}{\mathbf {ParOrd}}

\DeclareMathOperator{\Aut}{Aut}
\newcommand{\pre}{\mathbf{Preord}}

\DeclareMathOperator{\sgn}{sgn}

\newcommand{\cmat}{\left(\begin{array}}
\newcommand{\fmat}{\end{array}\right)}
\DeclareMathOperator{\Ob}{Ob}

\usepackage{stackrel}

\title[An extension of properties \dots category of mappings]{An extension of properties of symmetric group to monoids and a pretorsion theory in the category of mappings}
  \author[Alberto Facchini]{Alberto Facchini}
\address{Dipartimento di Matematica ``Tullio Levi-Civita'', Universit\`a di Padova,\linebreak 35121 Padova, Italy}
 \email{facchini@math.unipd.it}
\thanks{The first author is partially supported by Dipartimento di Matematica ``Tullio Levi-Civita'' of Universit\`a di Padova (Project BIRD163492/16 ``Categorical homological methods in the study of algebraic structures'' and Research program DOR1828909 ``Anelli e categorie di moduli''). }
\author[Leila Heidari Zadeh]{Leila Heidari Zadeh}
\address{Department of Mathematics, University of Kurdistan, P. O. Box 416,\linebreak Sanandaj, Iran.}
 \email{heidaryzadehleila@yahoo.com; l.heidaryzadeh@sci.uok.ac.ir}
   \keywords{Symmetric group, Full transformation monoid, Pseudoforest, Pretorsion theory. \\ \protect \indent 2010 {\it Mathematics Subject Classification.} Primary 18A99, 20M20.
Secondary 08A60.} 
  \begin{document}
\begin{abstract}
Several elementary properties of the symmetric group $S_n$ extend in a nice way to the full transformation monoid $M_n$ of all maps of the set $X:=\{1,2,3,\dots,n\}$ into  itself. The group $S_n$ turns out to be in some sense the torsion part of the monoid $M_n$. More precisely, there is a pretorsion theory in the category of all maps $f\colon X\to X$, $X$ an arbitrary finite non-empty set, in which bijections are exactly the torsion objects.\end{abstract}
 \maketitle
    
\section{Introduction}\label{s:1}
In all this paper,  $n\ge 1$ denotes a fixed integer and $X$ is the set \linebreak$\{1,2,3,\dots,n\}$. We teach every year to our first year students that:

(1) Every permutation can be written as a product of disjoint cycles, in a unique way up to the order of the factors. (We prove this associating a graph to every permutation, so that the decomposition as a product of pairwise disjoint cycles of the permutation follows from the partition of the graph into its connected components).

(2) Disjoint cycles permute.

(3) Every permutation can be written as a product of transpositions.

(4) Let $S_n$ be  the symmetric group, i.e., the group of all permutations of $X$. There is a group morphism $\sgn\colon S_n\to\{1,-1\}$. For every permutation $f\in S_n$, the number $\sgn(f)$ is called the  {\em sign} of the permutation $f$. 

(5) As a consequence, the group  $S_n$ of permutations has a normal subgroup $A_n$, the alternating subgroup, which is a subgroup of index $2$ of $S_n$ when $n\ge 2$.
It follows that, for $n\ge 2$, $S_n$ is the semidirect product of $A_n$ and any subgroup of $S_n$ generated by a trasposition.

In this paper, we develop the naive idea of extending the five results above from permutations $X\to X$ to arbitrary mappings $X\to X$. That is, we generalize the five results above from the group $S_n$ to the monoid $M_n$ of all mappings $X\to X$. The operation on $M_n$ is the composition of mappings. The group $S_n$ has order $n!$, while the monoid $M_n$ has $n^n$ elements. We find that:

(1) Every mapping $f\colon X\to X$ can be written as a product of ``forests on a cycle'', in a unique way up to the order of the factors. (This corresponds to the partition into connected components of an undirected graph $G^u_f$ associated to the mapping~$f$). The decomposition of $f$ as a product of disjoint forests on a cycle corresponds exactly to the decomposition of a permutation $\sigma$ as a product of disjoint cycles.

(2) Any two disjoint forests on a cycle permute.

(3) Any forest on a cycle can be written as a product of moves and transpositions. A move is a mapping $X\to X$ that fixes all elements of $X$ except for one. It is an idempotent mapping $X\to X$.

(4) Let $M_n$ be  the monoid of all mappings $X\to X$. Then $M_n$ is the disjoint union of its groups of units $S_n$ and the completely prime two-sided ideal $I_n$ of $M_n$ consisting of all non-injective mappings $X\to X$. As a consequnce, there is a monoid morphism $\sgn\colon M_n\to\{0,1,-1\}$ of the monoid $M_n$ into the multiplicative monoid $\{0,1,-1\}$. The mappings $f\in M_n$ with $\sgn(f)=1$ are exactly those in the alternating group $A_n$, the mappings $f$ with $\sgn(f)=-1$ are those in the coset $S_n\setminus A_n$ of $S_n$, and the  mappings $f\in M_n$ with $\sgn(f)=0$ are those in the two-sided ideal $I_n$.

(5) The subsemigroup $I_n$ of $M_n$ is generated by the set of all moves in $M_n$. The submonoid $I'_n:=I_n\cup\{1_{M_n}\}$ generated by the set of all moves is such that $M_n=S_n\cup I'_n$ and $S_n\cap I'_n= \{1_{M_n}\}$. There is a canonical epimorphism of the semidirect product $I'_n\rtimes S_n$ of $I'_n$ and $S_n$ onto $M_n$. 

\medskip

In the second part of the paper, we consider the category $\Cal M$ whose objects $(X,f)$ are all finite nonempty sets $X$ with a mapping $f\colon X\to X$. We show that in $\Cal M$ there is a very nice and interesting pretorsion theory $(\Cal C,\Cal F)$ in the sense of \cite{AC}. The pretorsion class $\Cal C$ consists of all objects $(X,f)$ of $\Cal M$ with $f$ a bijection, and $\Cal F$ consists of all objects $(X,f)$  of $\Cal M$ for which the graph $G^u_f$ associated to $f$ is a forest (equivalently, of all objects $(X,f)$ of $\Cal M$ with $f^n=f^{n+1}$).

\medskip

In this paper, all undirected graphs are simple and don't have loops, while directed graphs are simple, but can have loops.

\section{A description of mappings $X\to X$.}

Recall that, in all the paper, $X:=\{1,2,3,\dots,n\}$ for some positive integer $n$.

\begin{proposition}\label{leila} Let $f\colon X\to X$ be a mapping. Then there exist an integer $m\ge0$ and a partition $\{\,A_i\mid i=0,1,2,\dots, m\,\}$ of $X$ such that $f(A_0)=A_0$ and $f(A_i)\subseteq A_{i-1}$ for every $ i=1,2,\dots,m$. \end{proposition}

\begin{proof} Consider the descending chain $X\supseteq f(X) \supseteq f^2(X) \supseteq \dots$ of subsets of~$X$. This descending chain is stationary, hence there exists a least index $t\ge0$ with $f(f^t(X))=f^t(X)$. Set $A_0:=f^t(X)$, so $f(A_0)=A_0$. Now define an ascending chain of subsets $B_i$ of $X$, $i\ge0$, setting $B_0:=A_0$, and $B_i=f^{-1}(B_{i-1})$ for every $i>0$. Then $B_i\supseteq B_{i-1}$ for every $i>0$ (Induction on $i>0$. For $i=1$:  $f(A_0)=A_0$, so that $A_0\subseteq f^{-1}(A_0)$, i.e., $B_0\subseteq B_1$. Suppose the inclusion true for $i$, i.e., $B_i\supseteq B_{i-1}$. Then $B_{i+1}=f^{-1}(B_i)\supseteq
f^{-1}(B_{i-1})=B_i$). Also, $B_i\supseteq f^{t-i}(X)$ for every $i=0,1,2,\dots,t$, as can be easily seen by induction on $i$. Hence $B_t=X$. Thus we have an ascending chain $B_0\subseteq B_1\subseteq B_2\subseteq \dots\subseteq B_t=X$ of subsets of $X$. Let $m\ge0$ be the smallest integer $\ge 0$ with $B_m=X$. Then the chain $B_0\subset B_1\subset B_2\subset \dots\subset B_m=X$ is strictly ascending. Set $A_i:=B_i\setminus B_{i-1}$ for every $i=1,2,\dots, m$, so that $\{\,A_i\mid i=0,1,2,\dots, m\,\}$ is a partition of $X$. 

It remains to prove that $f(A_i)\subseteq A_{i-1}$ for every $ i=1,2,\dots,m$. We have that $B_i=f^{-1}(B_{i-1})$ for every $i>0$. 
Thus,  if $x\in A_i$, then $x\in B_i$ and $x\notin B_{i-1}$, so $x\in f^{-1}(B_{i-1})$ and (for $i\ge 2$) $x\notin f^{-1}(B_{i-2})$. Therefore $f(x)\in B_{i-1}\setminus B_{i-2}=A_{i-1}$, as desired. 
\end{proof}

For any finite set  $X$, a mapping $X\to X$ is injective, if and only if it is surjective, if and only if it is bijection.
Therefore the role of $A_0$ in the previous proposition is completely different from the role of the other blocks $A_1,\dots, A_m$ of the partition: the restriction of $f$ to $A_0$ is a permutation of $A_0$, while, for $i>0$, $f$ maps $A_i$ into~$A_{i-1}$.

\section{The directed graph and the undirected graph associated to a mapping $X\to X$.}\label{XXX}

Recall that all undirected graphs in this paper don't have multiple edges and loops, while directed graphs don't have multiple edges, but can have loops.

Given any mapping $f\colon X\to X$, it is possible to associate to $f$ a directed graph $G^d_f=(X,E^d_f)$, called the {\em graph of the function $f$},  having $X$ as a set of vertices and $E^d_f:=\{\,(i,f(i))\mid i\in X\,\}$ as a set of arrows. Hence $G^d_f$ has $n$ vertices and $n$ arrows, one arrow from $i$ to $f(i)$ for every $i\in X$. In the  directed graph $G^d_f$, every vertex has outdegree $1$. In Graph Theory, a direct graph in which every vertex has outdegree $1$ is sometimes called a {\em directed pseudoforest}. The corresponding $m$, defined as in Proposition~\ref{leila}, will be called the $height$ of the pseudoforest.

Similarly, it is possible to associate to $f$ an undirected graph $G^u_f=(X,E^u_f)$ having $X$ as a set of vertices and $E^u_f:=\{\,\{i,f(i)\}\mid i\in X,\ f(i)\ne i\,\}$ as a set of (undirected) edges. The graph  $G^d_f$ also has $n$ vertices, but $\le n$ edges. This occurs because of the fixed points, that is, the elements $i\in X$ with $f(i)=i$, and because of traspositions, i.e., the pairs $i,j$ of distinct elements of $X$ with $f(i)=j$ and $f(j)=i$. 

We will now describe the connected components of the graph $G^u_f$. Fix a vertex $x_0\in X$. We will determine the connected component of $x_0$. The connected component of $x_0$ in $G^u_f$ consists of all vertices $x\in X$ for which there exists a path from $x_0$ to $x$ in $G^u_f$. In our graph $G^u_f$, the edges are all of the form $\{x,f(x)\}$, provided $x\ne f(x)$. Thus the vertices at a distance $\le 1$ from $x_0$ are exactly those in the set $\{x_0,f(x_0)\}\cup f^{-1}(x_0)$. Hence the connected component of $x_0$ is the closure of $\{x_0\}$ with respect to taking images and inverse images via $f$.
Starting from the fixed vertex $x_0\in X$, we can define a sequence of vertices $x_0,x_1,x_3,\dots$ in $X$ with $x_{t+1}=f(x_t)$ for every $t\ge0$. Since $X$ is finite, there exists a smallest index $t_0\ge0$ with $x_{t_0}=x_t$ for some index $t>t_0$. Let $t_1$ be the smallest index $t>t_0$ with $x_{t_0}=x_t$. Then, in the graph $G^d_f$,  there is a simple (=no repetitions of vertices and edges) directed cycle $x_{t_0}, x_{t_0+1},\dots,x_{t_1}$. Now recursively define an ascending chain of subsets $B_0\subseteq B_1\subseteq B_2\subseteq \dots$ of $X$ setting $B_0:=\{x_{t_0}, x_{t_0+1},,\dots,x_{t_1-1}\}$ (the set of vertices on the cycle, which can be also a cycle of length $1$) and $B_{i+1}:= f^{-1}(B_i)$. Since the chain of submodules is ascending and the set $X$ is finite, the chain is necessarily stationary, so that there exists an index $m$ with $B_{m+1}=B_m$. Equivalently, $B_m=  f^{-1}(B_m)$. Moreover, $f(B_i)\subseteq B_{i-1}$ for every $i\ge 1$. There is no edge between any vertex in $X\setminus B_m$ and any vertex in $B_m$. Also, any two vertices in $B_m$ are connected by a path in $G^u_f$. The vertex $x_0$ is in $B_m$, because $f^{t_0}(x_0)=x_{t_0}\in B_0$, so that $x_0\in (f^{t_0})^{-1}(B_0)\subseteq B_{t_0}\subseteq B_m$. This proves that $B_m$ is the connected component of $G^u_f$ containing $x_0$. 

The full subgraph of $G^d_f=(X,E^d_f)$ whose set of vertices is $B_m$ is a connected directed pseudoforest. We will call a connected directed pseudoforest a {\em forest on a cycle}. Its typical form is like in ~\figref{01}.

\begin{figure}
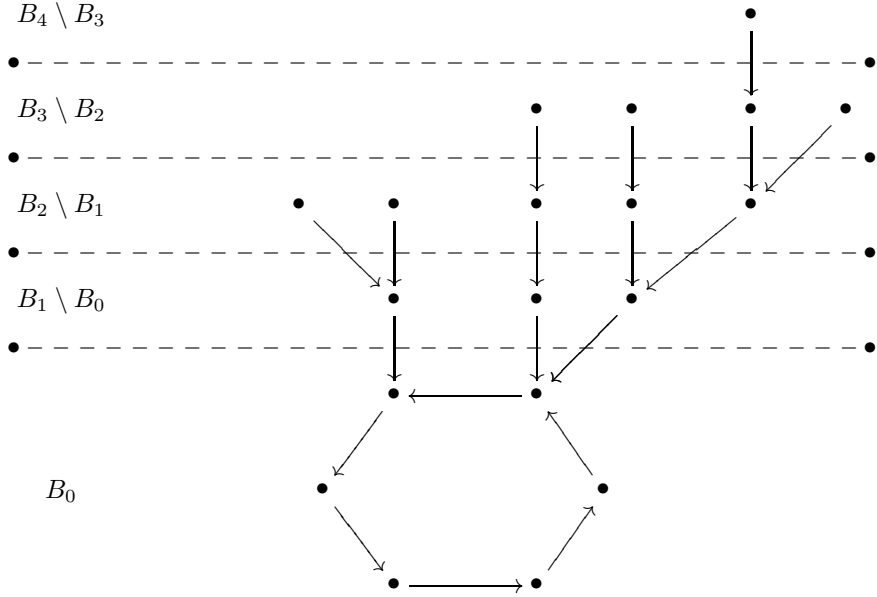

\[  \xygraph{
!{(2.25,1) }*+{\bullet_{}}="a"
!{(3,0) }*+{\bullet_{}}="b"
!{(4.5,0) }*+{\bullet_{}}="c"
!{(5.2,1)}*+{\bullet_{}}="d"
!{(4.5,2) }*+{\bullet_{}}="e"
!{(3,2) }*+{\bullet_{}}="f"
!{(3,3) }*+{\bullet_{}}="g"
!{(4.5,3)}*+{\bullet_{}}="h"
!{(5.5,3) }*+{\bullet_{}}="i"
!{(2,4) }*+{\bullet_{}}="j"
!{(3,4)}*+{\bullet_{}}="k"
!{(4.5,4) }*+{\bullet_{}}="l"
!{(5.5,4) }*+{\bullet_{}}="m"
!{(6.75,4) }*+{\bullet_{}}="n"
!{(4.5,5)}*+{\bullet_{}}="o"
!{(5.5,5) }*+{\bullet_{}}="p"
!{(6.75,5) }*+{\bullet_{}}="q"
!{(7.75,5) }*+{\bullet_{}}="r"
!{(6.75,6)}*+{\bullet_{}}="s"
!{(-.5,1) }*+{{B_0}}
!{(-.5,3) }*+{{B_1\setminus B_0}}
!{(-.5,4)}*+{{B_2\setminus B_1}}
!{(-.5,5) }*+{{B_3\setminus B_2}}
!{(-.5,6) }*+{{B_4\setminus B_3}}
!{(-1,2.5) }*+[white]{\bullet}="a1"
!{(-1,3.5) }*+[white]{\bullet}="a2"
!{(-1,4.5) }*+[white]{\bullet}="a3"
!{(-1,5.5) }*+[white]{\bullet}="a4"
!{(8,2.5) }*+[white]{\bullet}="z1"
!{(8,3.5) }*+[white]{\bullet}="z2"
!{(8,4.5) }*+[white]{\bullet}="z3"
!{(8,5.5) }*+[white]{\bullet}="z4"
"a":"b" "b":"c""c":"d"
"d":"e" "e":"f" "f":"a"
"j":"g" "g":"f""k":"g"
"o":"l" "l":"h" "h":"e"
"p":"m"
 "m":"i""i":"e"
"s":"q" "q":"n" "n":"i"
"i":"e""r":"n"
"a1"-@{--}"z1""a2"-@{--}"z2""a3"-@{--}"z3""a4"-@{--}"z4"
}  \]
\caption{\small A forest on a cycle, that is, a connected directed pseudoforest.}\label{01} 
\end{figure}

We will say that a function $f\colon X\to X$ is a {\em forest on a cycle} if there are a subset $A$ of $X$ and an element $x_0\in A$ such that: (1) for every $x\in A$ there exists an integer $t\ge 0$ with either $f^t(x_0)=x$ or $f^t(x)=x_0$, and (2) $f(y)=y$ for every $y\in X\setminus A$. Hence a bijection $f\colon X\to X$ is a forest on a cycle if and only if it is a cycle. 

We say that two functions $f,g\colon X\to X$ are {\em disjoint} if, for every $x\in X$, either $f(x)=x$ or $g(x)=x$ (or both).
Similarly to the fact that any two disjoint cycles commute, any two disjoint functions commute. In particular, any two disjoint forests on a cycle commute.

Since every undirected graph can be decomposed in a unique way into its connected components, in the same way as we see that every permutation can be written as a product of disjoint cycles in a unique way up to the order of the factors, we similarly get that:

\begin{theorem}\label{0} Every mapping $f\colon X\to X$ can be written as a product of disjoint forests on a cycle, in a unique way up to the order of the factors.\end{theorem}

This corresponds to the decomposition into connected components of the undirected graph $G^u_f$.

\section{Transpositions and moves.}\label{sss}

Every cycle, hence every permutation, is a product of transpositions. Let's see the analog for any mapping $f\colon X\to X$.

 A {\em move} is a mapping $f\colon X\to X$ such that there exists $x_0\in X$ with $f(x_0)\ne x_0$ and $f(x)=x$ for every $x\in X$, $x\ne x_0$. We will denote the move that maps $x_0$ to $x_1$ and fixes all the elements $x\in X$, $x\ne x_0$, by $m(x_0,x_1)$.
 
If $f\colon X\to X$ is any mapping, and $G^u_f$ is its associated pseudoforest, we will call {\em length} of $f$ the number of edges of $G^u_f$. Hence mappings of length $1$ are transpositions and moves.

\begin{theorem}\label{1} Let $f\colon X\to X$ be a mapping, $n=|X|$ and $p=|f^n(X)|=|A_0|$. Then $f$
can be written as a product $$f=m_1m_2\dots m_{n-p}t_1t_2\dots t_s,$$ where:\begin{enumerate}\item $m_1,m_2,\dots, m_{n-p}$ are $n-p$ moves, one for each element of $X\setminus A_0$, which fix all the elements of $A_0=f^n(X)$, and \item $t_1,t_2,\dots ,t_s$ are $s\ge 0$ transpositions of elements of $A_0$.\end{enumerate}\end{theorem}

\begin{proof} Let $\{\,A_i\mid i=0,1,2,\dots, m\,\}$ be the partition of $X$ as in Proposition~\ref{leila}. Then $A_0$ is invariant for $f$, and the restriction $f|_{A_0}^{A_0}\colon {A_0}\to {A_0}$ is a bijection. Extending this bijection to a bijection $\sigma\colon X\to X$ that is the identity on $X\setminus A_0$, we get an element $\sigma\in S_n$. Write $\sigma$ as a product $\sigma=t_1\circ t_2\circ \dots\circ  t_s$ of $s\ge0$ transpositions of elements of $A_0$. For every $i=1,2,\dots,m$, let $x_{i1},x_{i2},\dots,x_{in_i}$ be the elements of $A_i$, where $n_i$ is the cardinality of $A_i$. Then the required decomposition of $f$ is $$\begin{array}{l}f=m(x_{m1},f(x_{m1}))\circ m(x_{m2},f(x_{m2}))\circ \dots\circ m(x_{mn_m},f(x_{mn_m}))\circ \\ \qquad{}\circ 
m(x_{(m-1)1},f(x_{(m-1)1}))\circ \dots\circ m(x_{(m-1)n_{m-1}},f(x_{(m-1)n_{m-1}}))\circ \dots\circ \\ \qquad{}\circ 
m(x_{11},f(x_{11}))\circ m(x_{12},f(x_{12}))\circ{} \dots{}\circ m(x_{1n_1},f(x_{1n_1}))\circ t_1\circ t_2\circ\dots\circ t_s.\end{array}$$
\end{proof}

\begin{remark}{\rm \noindent (1) In the statement of Theorem~\ref{1}, the two composite mappings $m:=m_1m_2\dots m_{n-p}$ and $\sigma:=t_1t_2\dots t_s$ are completely determined by $f$. 
So $f=m\sigma$, where $m$ is a product of moves and $\sigma\in S_n$ is a permutation. Notice that, in general, a decomposition $f=g\tau$  of a mapping $f$ as a product $g$ of moves and a product $\tau$ of transpositions is not unique. For instance, $m(1,2)\circ (1\ 2)=m(1,2)$. Hence the additional conditions ``the moves, one for each element of $X\setminus A_0$, fix all the elements of $ A_0$, and the transpositions $t_1,t_2,\dots ,t_s$ act on elements of $A_0$'' in the statement of Theorem~\ref{1} are necessary.

\noindent (2) If we have two product decompositions $f=m\sigma$ and $f'=m'\sigma'$ like in the statement of Theorem \ref{1}, then $ff'=(m\sigma)(m'\sigma')=(m\sigma m'\sigma^{-1})(\sigma\sigma')$, and $m\sigma m'\sigma^{-1}$ is a product of moves, because the conjugate $\sigma m(i,
 j)\sigma^{-1}$ of any move $m(i,\ j)$ is the move $m(\sigma(i), \sigma(j))$, as is easily verified. But,  in the decomposition $ff'=(m\sigma m'\sigma^{-1})(\sigma\sigma')$, the additional conditions ``the moves, one for each element of $X\setminus A_0$, fix all the elements of $ A_0$, and the transpositions act on elements of $A_0$'' in the statement of Theorem~\ref{1} do not hold.}\end{remark}

\section{Idempotent mappings, products of moves, the sign $\sgn\colon M_n\to\{0,1,-1\}$, semidirect products.}

Any move is an idempotent mapping, that is, a mapping $f\colon X\to X$ such that $f^2=f$ (where $f^2=f\circ f$). Let us determine the structure of idempotent mappings. Let $f\colon X\to X$ be an idempotent mapping and $\{\,A_i\mid i=0,1,2,\dots, m\,\}$ the corresponding partition of $X$ according to Proposition~\ref{leila}. Then $f(A_i)\subseteq A_{i-1}$ for every $ i=1,2,\dots,m$. Suppose $m\ge 2$. Since $f^2(A_m)=f(A_m)$ is non-empty, $f^2(A_m)\subseteq A_{m-2}$, $f(A_m)\subseteq A_{m-1}$ and $A_{m-1},A_{m-2}$ are disjoint, we get a contradiction. Thus $m\le 1$. If $m=0$, then $f$ is an idempotent permutation, hence $f=1_{M_n}$, and its corresponding graph consists of $n$ disjoint vertices. If $m=1$, then we have the partition $\{A_0,A_1\}$ of $X$, $f(A_1)\subseteq A_0$, and, from $f^2(x)=f(x)$ for every $x\in X$, we get that $f$ is the identity on $A_0$. This shows that:

\begin{proposition} A mapping $f\colon X\to X$ is idempotent if and only if its graph is a forest of height at most $1$, if and only if there is a subset $A_0$ of $X$ such that $f(X\setminus A_0)\subseteq A_0$ and the restriction of $f$ to $A_0$ is the identity of $A_0$.\end{proposition}

\begin{figure}
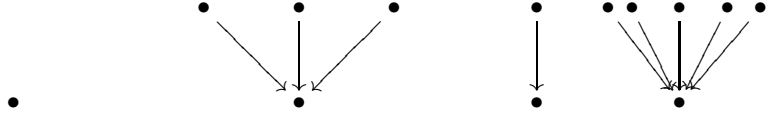
 
\[  \xygraph{
!{(-3,1) }*+{\bullet}="e"
!{(0,1)}*+{\bullet}="o"
!{(0,2) }*+{\bullet}="p"
!{(1,2) }*+{\bullet}="r"
!{(-1,2) }*+{\bullet}="r1"
!{(2.5,1)}*+{\bullet}="s"
!{(2.5,2)}*+{\bullet}="s2"
!{(4,1) }*+{\bullet}="v"
!{(4,2) }*+{\bullet}="w"
!{(3.5,2)}*+{\bullet}="x"
!{(3.25,2) }*+{\bullet}="y"
!{(4.5,2) }*+{\bullet}="z"
!{(4.85,2) }*+{\bullet}="z2"
"s2":"s"
"p":"o" 
"r":"o"
 "r1":"o"
 "w":"v"
"x":"v" "y":"v" "z":"v""z2":"v"
}  \]
\caption{\small An idempotent mapping: a forest of height at most 1.}\label{2}
\end{figure}

As a consequence, we have that in the product representation of Theorem~\ref{1} of an idempotent mapping $f\colon X\to X$ as a product $f=m_1m_2\dots m_{n-p}$ of moves (no transpositions are necessary), the moves $m_i$ commute pairwise. 

\medskip

Let $I_n$ be the subset of $M_n$ consisting of all non-injective mappings $X\to X$ (equivalently, all non-surjective mappings  $X\to X$). Clearly, for every $f,g\in M_n$, $fg\in I_n$ if and only if either $f\in I_n$ or $g\in I_n$, that is, $I_n$ is a completely prime two-sided ideal of the monoid $M_n$. By \cite{H}, every mapping in $I_n$ is a product of finitely many idempotent mappings. Since every idempotent mapping is a product of moves (Theorem~\ref{1}), it follows that the set of all moves in $M_n$ is a set of generators of the subsemigroup $I_n$ of $M_n$. 

\begin{remark}{\rm The fact that $I_n=M_n\setminus S_n$ is the subsemigroup $M_n$ generated by the set of all moves, implies that all elements $f\in I_n$ are product of moves, and all the other elements $f\in S_n=M_n\setminus I_n$ are product of transpositions.}\end{remark}

\medskip

Now the monoid $M_n$  is the disjoint union of its group of units $S_n=U(M_n)$ and the completely prime two-sided ideal $I_n$. 
As a consequence,  it is possible to define the sign of any mapping $f\colon X\to X$, which will be one of the three integers $1,-1$ and $0$. 
There is a monoid morphism $\sgn\colon M_n\to\{0,1,-1\}$ of the monoid $M_n$ into the multiplicative monoid $\{0,1,-1\}$.  It is defined, for every $f\in M_n$, by
\begin{equation*}
    \sgn(f)=
    \begin{cases}
      1&   \text{if }f\in A_n, \\
      -1&   \text{if }f\in S_n\setminus A_n, \\
      0&  \text{if }f\in I_n.
    \end{cases}
\end{equation*}

 Hence, for $n\ge 2$, there is an equivalence relation $\sim$ on the monoid $M_n$, compatible with the operation of $M_n$, whose equivalence classes are the three classes $A_n,\ S_n\setminus A_n$ and $I_n$, respectively.

\bigskip

For the symmetric group $S_n$, the sign $\sgn\colon S_n\to \{1,-1\}$ is a surjective group morphism  when $n\ge 2$, so that $S_n$ has a normal subgroup $A_n$, the alternating subgroup, which is a subgroup of index $2$ of $S_n$. Fix any transposition in $S_n$, for instance the transposition $(1\ 2)$. Then there is an endomorphism of $S_n$ that maps all even permutations of $S_n$ to the identity of $X$ and maps all odd permutations of $S_n$ to  the transposition $(1\ 2)$. This endomorphism of $S_n$ is an idempotent endomorphism of $S_n$. For groups, the existence of such an idempotent endomorphism is sufficient to have a splitting as a semidirect product: if $G$ is any group and $\varphi\colon G\to G$ is any idempotent endomorphism of $G$, then $G$ is the semidirect product of the kernel and the image of $\varphi$. 
Therefore, for $n\ge 2$, we have that $S_n=A_n\rtimes \langle (1\ 2)\rangle$, i.e., $S_n$ splits as a semidirect product of its normal subgroups $A_n$ and the subgroup $\langle (1\ 2)\rangle$ of order $2$. The group $\langle (1\ 2)\rangle$ acts on $A_n$ via conjugation. 

Let us see how this generalizes to the monoid $M_n$. 
The completely prime ideal $I_n$ of $M_n$ is a subsemigroup without identity of $M_n$. Let $I'_n$ be the submonoid of $M_n$ obtained from $I_n$ adjoining to it the identity of $M_n$, i.e., $I'_n:=I_n\cup\{1_{M_n}\}$ Then $I'_n$ is the submonoid of $M_n$ generated by the set of all moves, $M_n=S_n\cup I'_n$ and $S_n\cap I'_n= \{1_{M_n}\}$.

There is an action of $S_n$ on the normal submonoid $I'_n$ of $M_n$, i.e., a group homomorphism $S_n\to\Aut(I'_n)$ that maps any $\tau\in S_n$ to the inner automorphism $g\in I'_n\mapsto \tau g\tau^{-1}$ of the monoid $I'_n$. Hence it is possible to construct the semidirect product $ I_n'\rtimes S_n$ \cite{U}. There is a canonical surjective monoid homomorphism $$\psi\colon I_n'\rtimes S_n\to M_n, \quad\psi\colon (g,\tau)\mapsto g\tau,$$ because multiplication in the semidirect product is such that $$(g,\tau)(g',\tau')=(g\tau g'\tau^{-1}, \tau\tau').$$ If we consider the inverse image via $\psi$ of the elements of $M_n$, we find that the inverse image of an element $f\in M_n$ is $\varphi^{-1}(f)=\{\,(g,\tau)\mid g\in I_n', \tau\in S_n, g\tau=f\,\}$, which has cardinality $1$ if $f\in S_n$ and has cardinality $n!$ if $f\in M_n\setminus S_n$.

Notice that there is no idempotent endomorphism $\varphi$ of $M_n$ whose image is $S_n$. This is because a monoid morphism maps idempotents to idempotents. All moves are idempotents of $M_n$, and the only idempotent in $S_n$ is the identity. Thus we would have that $\varphi$ should map all moves of $M_n$ to the identity of $S_n$ and fix the elements of $S_n$. The equality $m(1,2)\circ (1\ 2)=m(1,2)$ seen in Section~\ref{sss} shows that this is not possible. 

The monoid $I_n'\rtimes S_n$ contains as a submonoid $I_n'\rtimes A_n$, on which the group $\langle(1\ 2)\rangle$ acts via conjugation. Hence it is possible to construct the semidirect product $(I'_n\rtimes A_n)\rtimes\langle(1\ 2)\rangle$, and there is a canonical epimorphism $p\colon (I'_n\rtimes A_n)\rtimes\langle(1\ 2)\rangle\to M_n$. The composite morphism $\sgn \circ p\colon (I'_n\rtimes A_n)\rtimes\langle(1\ 2)\rangle\to \{1,-1,0\}$ is the product $(\iota\rtimes\alpha)\rtimes\tau$ of the three morphisms:\enumerate\item $\iota\colon I'_n\to \{1,-1,0\}$ that maps all elements of $I_n$ to $0$ and the identity of $I'_n$ to $1$. \item  $\alpha\colon A_n\to \{1,-1,0\}$ that maps all elements of $A_n$ to $1$. And \item the injective monoid morphism $\tau\colon \langle(1\ 2)\rangle\to \{1,-1,0\}$ that maps the transposition $(1\ 2)$ to $-1$ and the identity of $\langle(1\ 2)\rangle$ to $1$: \[\xymatrix{ 
(I'_n\rtimes A_n)\rtimes\langle(1\ 2)\rangle \ar@{->}[r]^{\ \ \ \ \ \ \ \ \ p} \ar[dr]_{(\iota\rtimes\alpha)\rtimes\tau} & M_n \ar[d]^{\sgn} \\ 
 & {\phantom{.}}\{1,-1,0\}.
}\]

\section{Forests, cycles, and trivial mappings.}\label{pre}

\begin{proposition} The following conditions are equivalent for a mapping $f\colon X\to X$ with $|X|=n$:

{\rm (a)} The graph of $f$ is a forest.

{\rm (b)} $f^n=f^{n+1}$. \end{proposition}

\begin{proof} The graph of $f$ is a forest if and only if every connected component of the graph of $f$ is a tree, if and only if, for every $x_0\in X$, the connected component of $x_0$ has no cycles. In the notation of Section~\ref{XXX}, this means that $t_1=t_0+1$ for every vertex $x_0$ in $X$. Hence, equivalently, the graph of $f$ is a forest if and only for every $x\in X$ there is a $t_x\ge0$ such that $f^{t_x}(x)=f^{t_x+1}(x)$. Without loss of generality, we can suppose $t_x\le n$ because $|X|=n$. Thus the graph of $f$ is a forest if and only $f^n(x)=f^{n+1}(x)$ for every $x\in X$, that is, $f^n=f^{n+1}$.\end{proof}

\begin{proposition} The following conditions are equivalent for a mapping $f\colon X\to X$ with $|X|=n$:

{\rm (a)} $f$ is a bijection.

{\rm (b)} $f^{n!}$ is the identity $\iota_X\colon X\to X$. \end{proposition}

\begin{proof} If $f$ is a bijection, it is in the group of units $S_n$ of the monoid $M_n$, and $S_n$ has order $n!$, so that $f^{n!}$ is the identity $\iota_X\colon X\to X$ of the group~$S_n$.

Conversely, if $n=1$, then $f$ is certainly a bijection $X\to X$. If $n>1$, then $n!>1$, so that $f^{n!}=\iota_X$ implies that $f^{n!-1}$ is the inverse of $f$. Thus $f$ is a bijection.\end{proof}

Let $\Cal M$ be the category whose objects are all pairs $(X,f)$, where $X=\{1,2,3,\dots,$ $n\}$ for some $n\ge1$ and $f\colon X\to X$ is a mapping. Hence $\Cal M$ will be a small category with countably many objects. A morphism $g\colon (X,f)\to (X',f')$ in $\Cal M$ is any mapping $g\colon X\to X'$ for which the diagram  \begin{equation}
\xymatrix{ 
X \ar[r]^{g} \ar[d]_{f} & X' \ar[d]^{f'} \\ 
X\ar[r]_{g} & X'
}\label{homo}
\end{equation} commutes.

\bigskip

\begin{remark} {\rm Our category $\Cal M$ can also be seen from the point of view of Universal Algebra. It is equivalent to the category  (variety) of all finite algebras $(X,f)$ with one unary operation $f$ and no axioms. The morphisms in the category $\Cal M$ are exactly the homomorphisms in the sense of Universal Algebra. The product decomposition of $f$ as a product of disjoint forests on a cycle corresponds to the coproduct decomposition in this category $\Cal M$ as a coproduct of indecomposable algebras. A {\em congruence} on $(X,f)$, in the sense of Universal Algebra, is an equivalence relation $\sim$ on the set $X$ such that, for all $x,y\in X$, $x\sim y$ implies $f(x)\sim f(y)$.}\end{remark}

\bigskip

Now let  $\Cal C$ be the full subcategory of $\Cal M$ whose objects are the pairs $(X,f)$ with $f\colon X\to X$ a bijection. Let $\Cal F$ be the full subcategory of $\Cal M$ whose objects are the pairs $(X,f)$ where $f$ is a mapping whose graph is a forest. (Here, $\Cal C$ stands for cycles and $\Cal F$ stands for forests, or torsion-free objects, as we will see.) Clearly, an object of $\Cal M$ is an object both in $\Cal C$ and in $\Cal F$ if and only if it is of the form $(X,\iota_X)$, where $\iota_X\colon X\to X$ is the identity mapping. We will call these objects $(X,\iota_X)$ the {\em trivial objects} of~$\Cal M$. Let  $\Triv$ be the full subcategory of $\Cal M$ whose objects are all trivial objects $(X,\iota_X)$.

Call a morphism $g\colon (X,f)\to (X',f')$ in $\Cal M$ {\em trival} if it factors through a trivial object. That is, if there exists a trivial object $(Y,\iota_Y)$ and morphisms $h\colon (X,f)\to (Y,\iota_Y)$ and $\ell\colon (Y,\iota_Y)\to (X',f')$ in $\Cal M$ such that $g=\ell h$.

\begin{Lemma}\label{MMM} Let $g\colon (X,f)\to (X',f')$ be a morphism in $\Cal M$. Then $g$ is a trivial morphism in $\Cal M$ if and only if $g$ is constant on the connected components of $G^u_f$ and the image of $g$ consists of elements of $X'$ fixed by~$f'$. \end{Lemma}

\begin{proof} Let $g$ be a trivial morphism in $\Cal M$, so that there is a commutative diagram  \begin{equation}
\xymatrix{ 
X \ar[r]^{h} \ar[d]_{f} & Y \ar[r]^{\ell} \ar@{=}[d] & X' \ar[d]^{f'} \\ 
X\ar[r]_{h} & Y  \ar[r]_{\ell} & X'    ,
}\label{PPP}
\end{equation} where $g=\ell h$. In order to show that $g$ is constant on the connected components of $G^u_f$, it suffices to prove that if $x\in X$, then $g(x)=g(f(x))$. Now the commutativity of the square on the left of diagram~(\ref{PPP}) yields that $h(x)=h(f(x))$, so that $g(x)=\ell (h(x))=\ell(h(f(x)))=g(f(x))$. In order to prove that the image of $g$ consists of elements of $X'$ fixed by $f'$, we must show that, for every $x\in X$, $g(x)=f'(g(x))$. The commutativity of the square on the right of diagram~(\ref{PPP}) tells us that $\ell=f'\circ \ell$. Hence, for every $x\in X$, $f'( g(x))=f'(\ell(h(x)))=\ell(h(x))=g(x)$.

Conversely, let $g\colon X\to X'$ be a mapping that is constant on the connected components of $G^u_f$ and whose image consists of elements of $X'$ fixed by $f'$. For every $x\in X$, the vertices $x$ and $f(x)$ are in the same connected component of $G^u_f$, so that $g(x)=g(f(x))$. Since the image of $g$ consists of elements of $X'$ fixed by $f'$, we get that $f'(g(x))=g(x)$. This proves that $g=gf$ and $f'g=g$. Now $g$ factors as $g=\varepsilon\circ g|^{g(X)}$, where $\varepsilon \colon g(X)\hookrightarrow X'$ is the inclusion and $g|^{g(X)}\colon X\to g(X)$ is the corestriction of $g$ to $g(X)$. Thus the equalities $g=gf$ and $f'g=g$ can be rewritten as 
$\varepsilon g|^{g(X)}=\varepsilon g|^{g(X)}f$ and $f'\varepsilon g|^{g(X)}=\varepsilon g|^{g(X)}$, from which 
$g|^{g(X)}=g|^{g(X)}f$ and $f'\varepsilon =\varepsilon $ because $\varepsilon$ is injective and $g|^{g(X)}$ is surjective. Thus the diagram  \[
\xymatrix{ 
X \ar[r]^{g|^{g(X)}} \ar[d]_{f} & g(X) \ar[r]^{\varepsilon} \ar@{=}[d] & X' \ar[d]^{f'} \\ 
X\ar[r]_{g|^{g(X)}} & g(X) \ar[r]_{\varepsilon} & X'    .
}
\] is commutative. Hence $\varepsilon$ and $g|^{g(X)}$ are morphisms in $\Cal M$, and $g=\varepsilon \circ g|^{g(X)}$ is a factorization of $g$ through the trivial object $(g(X),\iota_{g(X)})$.
\end{proof} 

We conclude this section with a lemma that will be often useful in the sequel.

\begin{Lemma}\label{arrows} Let $g\colon (X,f)\to (X',f')$ be a morphism in $\Cal M$. Then:\begin{enumerate}
\item The image of an arrow in $G^d_f$ is an arrow in $G^d_{f'}$.
\item The image of a directed path (a directed cycle) in $G^d_f$ is a directed path (a directed cycle) in $G^d_{f'}$.
\item The image of a connected component of $G^u_f$ is contained in a connected component of $G^u_{f'}$.\end{enumerate}\end{Lemma}

\begin{proof} (a) We must show that for an arbitrary arrow $x\to f(x)$ of $G^d_f$, the arrow $g(x)\to g(f(x))$ is in $G^d_{f'}$, i.e., that $g(f(x))=f'(g(x))$, which holds because of the commutativity of digram~(\ref{homo}). Statements (b) and (c) follow immediately from (a).\end{proof}

\begin{corollary}\label{trivial} If $(X,f)$ and $(X',f')$ are objects of $\Cal M$, where $f$ is a bijection and the graph of $f'$ is a forest, then every morphism $g\colon (X,f)\to (X',f')$ is trivial.\end{corollary}

\begin{proof} By Lemma~\ref{MMM}, we must show that if $C$ is a connected component of $G^u_f$, then  $g$ is constant on $C$ and $g(C)$ consists of an element of $X'$ fixed by $f'$. Since $f$ is a bijection, the connected component $C$ of $G^u_f$, that is, a cycle, corresponds to an oriented cycle in $G^d_f$. By Lemma~\ref{arrows}(b), the image of a directed cycle of $G^d_f$ via $g$ is a directed cycle in $G^d_{f'}$. But the directed cycles in $G^d_{f'}$ are only the trivial cycles of length $1$ based on the elements of $X'$ fixed by $f'$. This concludes the proof.\end{proof}

\section{Prekernels and precokernels}

Let $f\colon X\to X'$ be a morphism in $\Cal M$. We say that a morphism $k\colon K\to X$ in $\Cal M$ is a \emph{prekernel} of $f$ if the following properties hold: 
\begin{enumerate}
	\item $fk$ is a trivial morphism.
	\item Whenever $\ell \colon Y\to X$ is a morphism in $\Cal M$ and $f\ell$ is trivial, then there exists a unique morphism $\ell'\colon Y\to K$ in $\Cal M$ such that $\ell=k\ell'$. 
\end{enumerate}

A \emph{precokernel} of $f$ is a morphism $p\colon X'\to P$ such that:
\begin{enumerate}
	\item $pf$ is a trivial morphism.
	\item Whenever $h\colon X'\to Y$ is a morphism such that $h f$ is trivial, then there exists a unique morphism $h'\colon P\to Y$ with $h=h' p$.
\end{enumerate}

\bigskip

Not all morphisms in $\Cal M$ have a prekernel. For instance, let $f\colon X\to X$ be any fixed-point-free permutation of $X$, for example any cycle of length $n$. Then, for every trivial object $(Y,\iota_Y)$ of $\Cal M$, there are no morphisms  $(Y,\iota_Y)\to (X,f)$ in $\Cal M$. Hence there is no trivial morphism $(Z,h)\to (X,f)$ in $\Cal M$, for any object $(Z,h)$ of $\Cal M$. Therefore  for every object $(W,\ell)$ of $\Cal M$, all morphisms $(W,\ell)\to (X,f)$ have no prekernel.

\bigskip

Now let $(X,f)$ be an object of $\Cal M$. Let $\sim$ be the congruence on the universal algebra $(X,f)$ generated by the subset $A:=\{\,(f^n(x), f^{n+1}(x))\mid x\in X\,\}$ of $X\times X$. That is, $\sim$ is the intersection of all the congruences of the universal algebra $(X,f)$, viewed as subsets of $X\times X$, that contain the subset $A$. 

\begin{Lemma}\label{NNN} The following conditions are equivalent for any two elements\linebreak $x_1, x_2\in X$:

{\rm (a)} $x_1\sim x_2$.

{\rm (b)} there exist non-negative integers $t_1, t_2$ such that $x_1=f^{t_1}(x_2)$ and $x_2=f^{t_2}(x_1)$.

{\rm (c)} $x_1$ and $x_2$ are on an oriented cycle (possibly of length $1$) of $G^d_f$.\end{Lemma}

\begin{proof} (a)${}\Rightarrow{}$(b). Let $\sim'$ be the relation on $X$ defined, for every $x_1,x_2\in X$, by $x_1\sim' x_2$ if there exist non-negative integers $t_1, t_2$ such that $x_1=f^{t_1}(x_2)$ and $x_2=f^{t_2}(x_1)$. We want to show that $\sim\,\,\subseteq\,\,\sim'$. Now $\sim'$ is reflexive (take $t_1=t_2=0$), and clearly symmetric. In order to prove that it is transitive, suppose $x_1\sim' x_2$ and $x_2\sim' x_3$ ($x_1,x_2,x_3\in X$). There exist integers $t_1, t_2, t_3,t_4\ge0$ such that $x_1=f^{t_1}(x_2)$, $x_2=f^{t_2}(x_1)$, $x_2=f^{t_3}(x_3)$ and $x_3=f^{t_4}(x_2)$. Then $x_1=f^{t_1}(x_2)=f^{t_1+t_3}(x_3)$ and $x_3=f^{t_4}(x_2)=f^{t_4+t_2}(x_1)$. Hence $x_1\sim' x_3$. This proves that $\sim'$ is an equivalence relation on the set $X$. Moreover, if 
$x_1\sim' x_2$, then $x_1=f^{t_1}(x_2)$ and $x_2=f^{t_2}(x_1)$ for suitable $t_1,t_2\ge0$, so $f(x_1)=f^{t_1}(f(x_2))$ and $f(x_2)=f^{t_2}(f(x_1))$. Thus $f(x_1)\sim' f(x_2)$. This shows that $\sim'$ is a congruence for the universal algebra $(X,f)$. Now $A\subseteq\,\,\sim'$, because if $(f^n(x), f^{n+1}(x))$ is an arbitrary element of $A$ ($x\in X$), then, in the  notation of Section~\ref{XXX}, $f^{t_0}(x)=f^{t_1}(x)$, so $f^n(x)=f^{n+k}(x)$ for some $k>0$. Thus, for $t_1:=k-1$, we have that $f^n(x)=f^{n+k}(x)=f^{k-1}(f^{n+1}(x))=f^{t_1}(f^{n+1}(x))$ and, for $t_2:=1$, $f^{n+1}(x)=f^{t_2}(f^n(x))$. This proves that $f^n(x)\sim' f^{n+1}(x)$. We have thus shown that $A\subseteq\,\,\sim'$. As $\sim$ is generated by $A$, we obtain $\sim\,\,\subseteq\,\,\sim'$, that is, (a)${}\Rightarrow{}$(b).

(b)${}\Rightarrow{}$(a). We must show that $\sim'\,\subseteq\,\,\sim$. Suppose $x_1\sim' x_2$ ($x_1,x_2\in X$). If $x_1=x_2$, then $x_1\sim x_2$. Therefore we can suppose $x_1\ne x_2$. Thus there exist $t_1, t_2\ge0 $ such that $x_1=f^{t_1}(x_2)$, $x_2=f^{t_2}(x_1)$ and $t_1+t_2>0$. Hence $x_1=f^{t_1+t_2}(x_1)$. Consider the subset $\{x_1,f(x_1), f^2(x_1),\dots , f^{t_1+t_2-1}(x_1)\}$ of $X$. Then $f$ permutes cyclically the elements of this subset, so that $f^n$ permutes the elements of this subset 
\begin{align*}
&\{x_1,f(x_1), f^2(x_1),\dots , f^{t_1+t_2-1}(x_1)\}\\ &\qquad=\{f^n(x_1),f^{n+1}(x_1), f^{n+2}(x_1),\dots , f^{n+t_1+t_2-1}(x_1)\}.\end{align*}
 But $(f^n(x_1),f^{n+1}(x_1))\in A$, $(f^{n+1}(x_1), f^{n+2}(x_1))\in A$, and so on, so that $f^n(x_1)\sim f^{n+1}(x_1)$, $f^{n+1}(x_1)\sim f^{n+2}(x_1)$, \dots, $f^{n+t_1+t_2-2}(x_1)\sim f^{n+t_1+t_2-1}(x_1)$. Thus 
all the elements of the subset $$\{f^n(x_1),f^{n+1}(x_1), f^{n+2}(x_1),\dots , f^{n+t_1+t_2-1}(x_1)\}$$ are equivalent modulo $\sim$.  But $x_1$ and $x_2=f^{t_2}(x_1)$ are two of those elements, so $x_1\sim x_2$.

(b)${}\Leftrightarrow{}$(c) is trivial.\end{proof}

Clearly, the canonical projection $\pi\colon X\to X/\!\!\sim$ is the coequalizer  of the two endomorphisms $f^n$ and $f^{n+1}$ of $(X,f)$. Also, since $\sim$ is a congruence  on the unary universal algebra $(X,f)$, $f$ induces a mapping $\overline{f}\colon X/\!\!\sim{}{}\to X/\!\!\sim$. Thus $(X/\!\!\sim,\overline{f})$ is a universal algebra and the diagram
\[\xymatrix{ 
X \ar[r]^{\pi\ } \ar[d]_{f} & X/\!\!\sim\ar[d]^{\overline{f}} \\ 
X\ar[r]_{\pi\ } & X/\!\!\sim
}
\] commutes. Notice that $X/\!\!\sim$ is a forest, because $\overline{f}^n=\overline{f}^{n+1}$, so that $\overline{f}$ always has fixed points (the roots of the trees in the forest $G^d_{\overline{f}}$). 

As an example, we have drawn the graph $G^d_f$ of a mapping $f\colon X\to X$ in ~\figref{3}. In this example, $X$ has $n=27$ elements. In ~\figref{4}, we have represented the partition of the set $X$ of vertices of the graph $G^d_f$ in equivalence classes modulo $\sim$. One equivalence classes has $6$ elements, another one has $4$ elements, and all the other classes consist of one element. Hence there are $19 $ equivalence classes. ~\figref{5} represents the graph of $(X/\!\!\sim,\overline{f})$, which is a graph with $19$ vertices.

\begin{figure}
\[  \xygraph{
!{(-4.75,2) }*+{\bullet}="a"
!{(-4,1) }*+{\bullet}="b"
!{(-3,1) }*+{\bullet}="c"
!{(-2.25,2)}*+{\bullet}="d"
!{(-3,3) }*+{\bullet}="e"
!{(-4,3) }*+{\bullet}="f"
!{(-4,4) }*+{\bullet}="g"
!{(-2,4)}*+{\bullet}="h"
!{(-4.75,5) }*+{\bullet}="i"
!{(-4,5) }*+{\bullet}="j"
!{(-3.25,5)}*+{\bullet}="k"
!{(-2,5) }*+{\bullet}="l"
!{(-1,4.9) }*+{\bullet}="m"
!{(-2,6) }*+{\bullet}="n"
!{(0,1)}*+{\bullet}="o"
!{(0,2) }*+{\bullet}="p"
!{(1,1) }*+{\bullet}="q"
!{(1,2) }*+{\bullet}="r"
!{(2.5,1)}*+{\bullet}="s"
!{(4,1)}*+{\bullet}="t"
!{(4,2) }*+{\bullet}="u"
!{(4,3) }*+{\bullet}="v"
!{(4,4) }*+{\bullet}="w"
!{(3.5,4)}*+{\bullet}="x"
!{(3.25,4) }*+{\bullet}="y"
!{(4.5,4) }*+{\bullet}="z"
!{(4.85,4) }*+{\bullet}="z2"
"a":"b" "b":"c""c":"d""i":"g"
"d":"e" "e":"f" "f":"a"
"j":"g" "g":"f""k":"g"
"n":"l" "l":"h" "h":"e"
"m":"h""o":"q"
"q":"r" "r":"p" "p":"o"
"u":"t" "v":"u" "w":"v"
"x":"v" "y":"v" "z":"v""z2":"v"
}  \] \caption{\small The graph $G^d_f$ of a mapping $f\colon X\to X$.}\label{3} 
\end{figure}

\begin{figure}
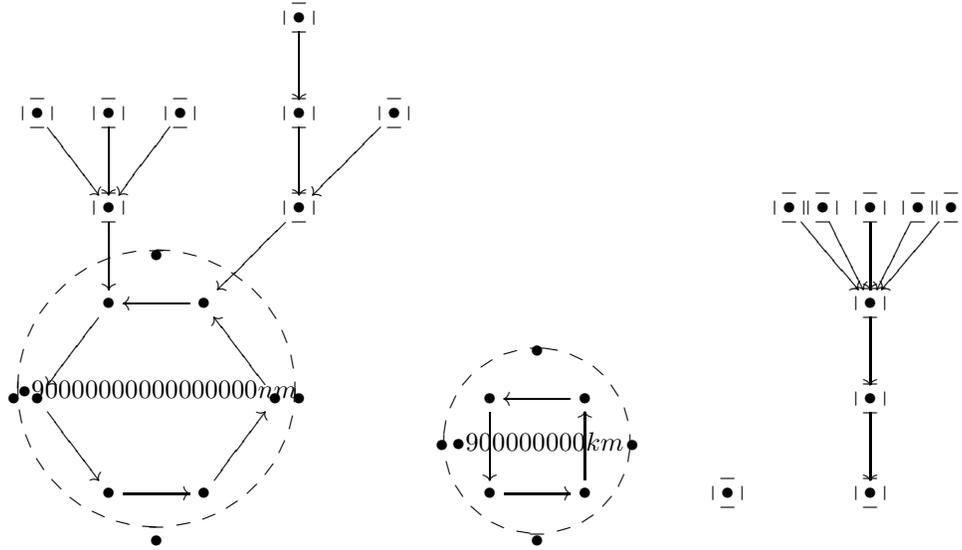

\[  \xygraph{
!{(-4.75,2) }*+{\bullet}="a"
!{(-4,1) }*+{\bullet}="b"
!{(-3,1) }*+{\bullet}="c"
!{(-2.25,2)}*+{\bullet}="d"
!{(-3,3) }*+{\bullet}="e"
!{(-4,3) }*+{\bullet}="f"
!{(-4,4) }*+{\bullet}*\frm{-o}="g"
!{(-2,4)}*+{\bullet}*\frm{-o}="h"
!{(-4.75,5) }*+{\bullet}*\frm{-o}="i"
!{(-4,5) }*+{\bullet}*\frm{-o}="j"
!{(-3.25,5)}*+{\bullet}*\frm{-o}="k"
!{(-2,5) }*+{\bullet}*\frm{-o}="l"
!{(-1,5) }*+{\bullet}*\frm{-o}="m"
!{(-2,6) }*+{\bullet}*\frm{-o}="n"
!{(0,1)}*+{\bullet}="o"
!{(0,2) }*+{\bullet}="p"
!{(1,1) }*+{\bullet}="q"
!{(1,2) }*+{\bullet}="r"
!{(2.5,1)}*+{\bullet}*\frm{-o}="s"
!{(4,1)}*+{\bullet}*\frm{-o}="t"
!{(4,2) }*+{\bullet}*\frm{-o}="u"
!{(4,3) }*+{\bullet}*\frm{-o}="v"
!{(4,4) }*+{\bullet}*\frm{-o}="w"
!{(3.5,4)}*+{\bullet}*\frm{-o}="x"
!{(3.15,4) }*+{\bullet}*\frm{-o}="y"
!{(4.5,4) }*+{\bullet}*\frm{-o}="z"
!{(4.85,4) }*+{\bullet}*\frm{-o}="z2"
!{(-5,2)}*+[white]{\bullet}="A"
!{(-3.5,.5) }*+[white]{\bullet}="B"
!{(-2,2) }*+[white]{\bullet}="C"
!{(-3.5,3.5) }*+[white]{\bullet}="D"
!{(.5,.5)}*+[white]{\bullet}="O"
!{(1.5,1.5) }*+[white]{\bullet}="P"
!{(-.5,1.5) }*+[white]{\bullet}="Q"
!{(.5,2.5) }*+[white]{\bullet}="R"
!{(.5,1.55)}*+[white]{\bullet900000000km}*\frm{-o}="O2"
!{(-3.5,2.099)}*[white]{\bullet90000000000000000nm}*\frm{-o}="O1"
"a":"b" "b":"c""c":"d"
"i":"g"
"d":"e" "e":"f" "f":"a"
"j":"g" "g":"f""k":"g"%
"n":"l" "l":"h" "h":"e"
"m":"h""o":"q"
"q":"r" "r":"p" "p":"o"
"u":"t" "v":"u" "w":"v"
"x":"v" "y":"v" "z":"v""z2":"v"
}  \] \caption{\small The partition of the set $X$ of vertices in equivalence classes modulo $\sim$.}\label{4} 
\end{figure}

\begin{figure}
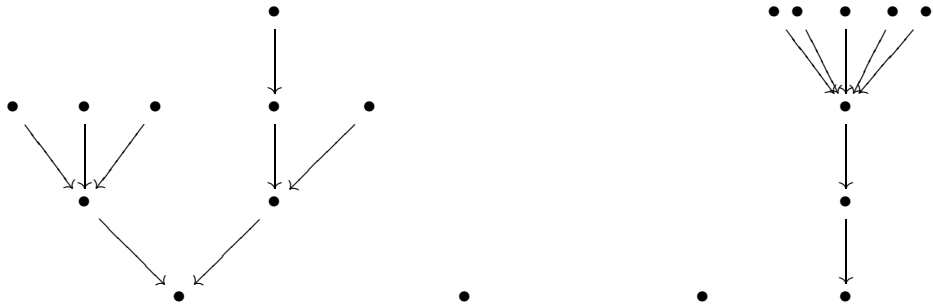
 
\[  \xygraph{
!{(-3,1) }*+{\bullet_{}}="e"
!{(-4,2) }*+{\bullet_{}}="g"
!{(-2,2)}*+{\bullet_{}}="h"
!{(-4.75,3) }*+{\bullet_{}}="i"
!{(-4,3) }*+{\bullet_{}}="j"
!{(-3.25,3)}*+{\bullet_{}}="k"
!{(-2,3) }*+{\bullet_{}}="l"
!{(-1,3) }*+{\bullet_{}}="m"
!{(-2,4) }*+{\bullet_{}}="n"
!{(0,1)}*+{\bullet_{}}="o"
!{(2.5,1)}*+{\bullet_{}}="s"
!{(4,1)}*+{\bullet_{}}="t"
!{(4,2) }*+{\bullet_{}}="u"
!{(4,3) }*+{\bullet_{}}="v"
!{(4,4) }*+{\bullet_{}}="w"
!{(3.5,4)}*+{\bullet_{}}="x"
!{(3.25,4) }*+{\bullet_{}}="y"
!{(4.5,4) }*+{\bullet_{}}="z"
!{(4.85,4) }*+{\bullet_{}}="z2"
"i":"g"
"j":"g"
 "g":"e" 
"k":"g"
"n":"l" "l":"h" "h":"e"
"m":"h"
"u":"t" "v":"u" "w":"v"
"x":"v" "y":"v" "z":"v""z2":"v"
}  \] \caption{\small The quotient set $X/\!\!\sim$.}\label{5}
\end{figure}

Let us prove that the morphism $\pi\colon X\to X/\!\!\sim$ has a prekernel, which is the embedding $\varepsilon\colon (A_0,f|^{A_0}_{A_0})\hookrightarrow(X,f)$. Here $A_0$ is the subset of $X$ as defined in Proposition~\ref{leila}. Hence the restriction $f|^{A_0}_{A_0}\colon A_0\to A_0$ of $f$ to $A_0$ is a bijection. Clearly, $\varepsilon$ is a morphism in $\Cal M$, $\pi\varepsilon\colon A_0\to X/\!\!\sim$ sends the cycles of $A_0$ to the roots of the tree $X/\!\!\sim$, hence $\pi\varepsilon$ is a trivial morphism by Lemma~\ref{MMM}. In order to prove property (2) in the definition of prekernel, fix a morphism $\ell \colon (Y,g)\to (X,f)$ in $\Cal M$ with $\pi\ell$  trivial. We have a commutative diagram \begin{equation*}
\xymatrix{ 
Y \ar[r]^{\ell} \ar[d]_{g} & X\ar[r]^{\pi} \ar[d]^{f} & X/\!\!\sim\ar[d]^{\overline{f}} \\ 
Y\ar[r]_{\ell} & X  \ar[r]^{\pi} & X/\!\!\sim   .
}
\end{equation*} Thus, for any connected component $C_Y$ of the graph $G^u_g$ of $(Y,g)$, we have that $\ell(C_Y)$ is contained in a cycle of $X$. It follows that $\ell(Y)\subseteq A_0$. Let $\ell':=\ell|^{A_0}\colon Y\to A_0$ be the corestriction of $\ell$, obtained by restricting the codomain $X$ of $\ell$ to $A_0$. Then $\ell=\varepsilon\ell'$. As far as the uniqueness of such an $\ell'$ is concerned, suppose that we also have another morphism $\ell''\colon Y\to A_0$ with $\ell=\varepsilon\ell''$. Then $\varepsilon\ell'=\varepsilon\ell''$, so $\ell'=\ell''$ because $\varepsilon$ is an injective mapping. This proves that the embedding $\varepsilon\colon (A_0,f|^{A_0}_{A_0})\hookrightarrow(X,f)$ is the prekernel of the canonical projection $\pi\colon X\to X/\!\!\sim$.

Conversely, we will now show that the canonical projection $\pi\colon X\to X/\!\!\sim$ is the precokernel of the embedding $\varepsilon\colon (A_0,f|^{A_0}_{A_0})\hookrightarrow(X,f)$. It is sufficient to prove property (2) in the definition of precokernel. Hence suppose we have a commutative diagram \begin{equation*}
\xymatrix{ 
A_0 \ar[r]^{\varepsilon} \ar[d]_{f|^{A_0}_{A_0}} & X\ar[r]^{\alpha} \ar[d]^{f} & Z\ar[d]^{g} \\ 
A_0\ar[r]_{\varepsilon} & X  \ar[r]^{\alpha} & Z
}
\end{equation*} with $\alpha\varepsilon$ trivial. Then the image $\alpha(C_f)$ of any cycle $C_f$ of $f$ is a fixed point of $(Z,g)$, that is, is one of the roots of a tree in the graph $G^u_g$ of $(Z,g)$. But if $x,y\in X$ and $x\sim y$, then $\alpha(x)=\alpha(y)$. Hence $\alpha$ factors through $\pi$: $$\xymatrix{ 
X\ar[r]^{\alpha} \ar[d]^{\pi} & Z\\ 
X/\!\!\sim.  \ar[ur]_{\alpha'} & 
}$$ The uniqueness of $\alpha'$ follows from the surjectivity of $\pi$.

\bigskip

We will recall the exact definition of pretorsion theory in Section~\ref{QQQ}, but the next result shows that the pair $(\Cal C,\Cal F)$ is a pretorsion theory in $\Cal M$.

\begin{theorem}\label{111} In the notation above, we have that:

{\rm (a)} For every object $(X,f)$ of $\Cal M$, there are morphisms \begin{equation}(A_0,f|^{A_0}_{A_0})\stackrel{\varepsilon}{\hookrightarrow}(X,f)\stackrel{\pi}{\hookrightarrow}(X/\!\!\sim,\overline{f})\label{_}\end{equation} in the category $\Cal M$ such that $\varepsilon$ is a prekernel of $\pi$ and $\pi$ is a precokernel of~$\varepsilon$.

{\rm (b)}  $\Hom_{\Cal M}(C,F)=\Triv_{\Cal M}(C,F)$ for every $C\in \Cal C ,\ F\in\mathcal F$. 

{\rm (c)}  If $X$ is an object of $\Cal M$ and $\Hom_{\Cal M}(X,F)=\Triv_{\Cal M}(X,F)$ for every $F\in\mathcal F$, then $X$ is an object of $\Cal C$.

{\rm (d)}  If $X$  is an object of $\Cal M$ and $\Hom_{\Cal M}(C,X)=\Triv_{\Cal M}(C,X)$ for every $C\in \Cal C$, then $X$ is an object of $\Cal F$. \end{theorem}

\begin{proof} (a) was proved in the discussion above, before the statement of the theorem.

(b) is Corollary~\ref{trivial}.

(c) Let $X$ be an object of $\Cal M$ with $\Hom_{\Cal M}(X,F)=\Triv_{\Cal M}(X,F)$ for every $F\in\mathcal F$. In (\ref{_}), we have that the morphism $\pi$ is trivial because $X/\!\!\sim$ is in $\Cal F$. Hence $\pi$ sends any connected component of $G^u_f$ to a root of the forest $(X/\!\!\sim,\overline{f})$. The inverse images via $\pi$ of the roots of the forest $(X/\!\!\sim,\overline{f})$ are the cycles of $(X,f)$. Hence any connected component of $G^u_f$ is contained in a cycle of $G^u_f$. Thus
any connected component of $G^u_f$ is a cycle of $G^u_f$. Therefore $f$ is a bijection.

(d) Finally, let $X$ be an object of $\Cal M$. Assume that $\Hom_{\Cal M}(C,X)=\Triv_{\Cal M}(C,X)$ for every $C\in \Cal C$.
Suppose that $X$ has a directed cycle, so that the cycle is contained in $A_0$ necessarily. We have the embedding $\varepsilon\colon A_0\to X$, where $(A_0,f|^{A_0}_{A_0})$ is an object of $\Cal C$. By the hypothesis $\Hom_{\Cal M}(A_0,X)=\Triv_{\Cal M}(A_0,X)$, we have that $\varepsilon$ is trivial, so that the inclusion $\varepsilon$ maps every cycle in $A_0$ to a root of~$X$. Thus every cycle consists of a single point, so that in $A_0$ there are no cycles  of length $>0$. Hence in $G^u_f$ there are no cycles. This proves that $X$ is an object of $\Cal F$.\end{proof}

\section{Some noteworthy functors from the category $\Cal M$ and adjoint functors of the embeddings.}

In this section, we are going to consider some remarkable functors. 

We call {\em preradical} of $\Cal M$ any functor $R\colon\Cal M\to\Cal M$ such that, for each object $(X,f)$ of $\Cal M$, $R(X,f)$ is a subalgebra of $(X,f)$ and, for every morphism $g\colon (X,f)\to (X',f')$, $R(g)\colon R(X,f)\to R(X',f')$ is the restriction $g|_{R(X,f)}^{R(X',f')}$ of $g$. Thus a preradical is a subfunctor of the identity functor on $\Cal M$. We say that a preradical $R$ is {\em idempotent} if $R\circ R=R$.

Dually, we call {\em precoradical} of $\Cal M$ any functor $C\colon\Cal M\to\Cal M$ such that, for each object $(X,f)$ of $\Cal M$, $C(X,f)$ is a quotient algebra $(X/\!\!\equiv, \overline{f})$ of $(X,f)$ for some congruence $\equiv$ on the algebra $(X,f)$ and, for every morphism $g\colon (X,f)\to (X',f')$, $C(g)\colon C(X,f)\to C(X',f')$ is induced by $g$. Hence a precoradical is a quotient functor of the identity functor on $\Cal M$. We say that a precoradical $C$ is {\em idempotent} if, for every object $(X,f)$ of $\Cal M$, $C(\pi)\colon C(X,f)\to C(C(X,f))$ is an isomorphism. Here $\pi\colon (X,f)\to C(X,f)=(X/\equiv, \overline{f})$ is the canonical projection.

The first noteworthy functor we study in this section is an idempotent functor $C\colon\Cal M\to\Cal M$ such that $C(X,f)=(X/\!\!\sim,\overline{f})\in\Ob(\Cal F)$ for every object $(X,f)$ of $\Cal M$. The notation is an in Lemma~\ref{NNN}. 

\begin{proposition} There is an idempotent precoradical $C\colon\Cal M\to\Cal M$ such that $C(X,f)=(X/\!\!\sim,\overline{f})\in\Ob(\Cal F)$ for every object $(X,f)$ of $\Cal M$. \end{proposition}

\begin{proof} In order to prove that $C\colon\Cal M\to\Cal M$ is a precoradical, we must show that every morphism $g\colon (X,f)\to (X,f')$ induces a morphism $\overline{g}\colon (X/\!\!\sim,$ $\overline{f})\to (X'/\!\!\sim',\overline{f'})$. By Lemma~\ref{arrows}(b),    $g$ maps any oriented cycle of $G^d_f$ to an oriented cycle of $G^d_{f'}$. From Lemma~\ref{NNN}, it follows that, for every $x_1,x_2\in X$, $x_1\sim x_2$ implies $g(x_1)\sim' g(x_2)$. Thus the kernel $\sim_{\pi'g}$ of $\pi'g$ contains the kernel $\sim_{\pi}$ of $\pi$. We have the following diagram: \[
\xymatrix{ 
X \ar[r]^{\pi} \ar[d]_{g} & X/\!\!\sim \ar@{.>}[d]^{\exists !\overline{g}} \\ 
X' \ar[r]_{\pi'} & X'/\!\!\sim'
}
\]
Hence there exists a unique mapping $\overline{g}\colon X/\!\!\sim\,\to X'/\!\!\sim'$ that makes the diagram commute. It is now clear that $C$ is a precoradical.

We have already seen in Section~\ref{pre} that $(X/\!\!\sim,\overline{f})$ is an object of $\Cal F$ for every object $(X,f)$ of $\Cal M$. Moreover, 
for every $(F,h)$ in $\Cal F$ the congruence $\sim$ on $F$ is the equality on $f$. It follows that the precoradical $C$ is idempotent.
\end{proof}

\begin{proposition} There is an idempotent preradical $R\colon\Cal M\to\Cal M$ such that $R(X,f)=(A_0,f|_{A_0}^{A_0})\in\Ob(\Cal C)$ for every object $(X,f)$ of $\Cal M$. Here $A_0=f^n(X)$ is the set of all vertices of $G^d_f$ on some oriented cycle.\end{proposition}

\begin{proof} Like for the previous proposition, to prove that $R\colon\Cal M\to\Cal M$ is a preradical, it suffices to show that every morphism $g\colon (X,f)\to (X,f')$ maps $A_{0,f}$ to $A_{0,f'}$. By Lemma~\ref{arrows}, $g$ maps any cycle  of $G^d_f$ to a cycle of $G^d_{f'}$. Therefore $g(A_{0,f})\subseteq A_{0,f'}$. The rest is clear.\end{proof}

Notice that the canonical mapping $(X,f)\to C(X,f)$ has a prekernel, which is $R(X,f)$ (Theorem~\ref{111}). The prekernel of the canonical embedding $R(X,f)\hookrightarrow (X,f)$ does not exists in general.

\bigskip

A third remarkable functor $W\colon\Cal M\to\Cal C$ is the following. Let $(X,f)$ be an object of $\Cal M$ and we write $f$ as $f=m\sigma$, where $m$ is a product of moves (of the vertices on the forest) and $\sigma$ is a permutation (of the vertices on the cycles), like in Theorem~\ref{1}. That is, $m$ is a product of moves of elements of $X\setminus A_0$ to elements of $X$ and $\sigma$ permutes the elements of $A_0$. Then $f^{n!}(x)\in A_0$ for every $x\in X$, so that $\sigma^{-(n!)}(f^{n!}(x))\in A_0$. 

\begin{Lemma} The mapping $w\colon (X,f)\to (A_0,\sigma)$, defined by $$w(x)=\sigma^{-(n!)}(f^{n!}(x))\qquad\mbox{\rm for every }x\in X,$$ is a morphism in $\Cal M$.\end{Lemma}

\begin{proof} We must prove that $\sigma w=w f$. Now $f^{n!}(x)\in A_0$ for every $x\in X$, and $f$ and $\sigma$ coincide on $A_0$, so that $ff^{n!}=\sigma f^{n!}$. Therefore $\sigma w=\sigma \sigma^{-(n!)}f^{n!}=\sigma^{-(n!)}\sigma f^{n!}=\sigma^{-(n!)}ff^{n!}=\sigma^{-(n!)}f^{n!}f=w f$.\end{proof}

Notice that $w$ fixes the points of $A_0$, i.e., the points on the cycles, and winds up to the trees of the forest around the cycles. The roots of the trees are fixed by $w$. The permutation $\sigma$ is the restriction of $f$ to $A_0$.

\smallskip

We are ready to study the existence of right and left adjoints of the embeddings $\Cal C\hookrightarrow\Cal M$ and $\Cal F\hookrightarrow\Cal M$. 

\begin{proposition} The category $\Cal C$ is a reflective and coreflective subcategory of $\Cal M$, that is, the embedding $\Cal C\hookrightarrow\Cal M$ has both a left adjoint and a right adjoint. The left adjoint and right adjoint is the functor restriction $R\colon \Cal M\to\Cal C$ that maps the object $(X,f)$ to its restriction $(A_0, \sigma)$.\end{proposition}

\begin{proof} In order to show that $R\colon \Cal M\to\Cal C$ is the left adjoint of the embedding $\Cal C\hookrightarrow\Cal M$, with unit the morphism $w\colon (X,f)\to (A_0,\sigma)$, it suffices to show that for every object $(X,f)$ of $\Cal M$, every object $(C,\tau)$ of $\Cal C$ and every morphism $\varphi\colon (X,f)\to (C,\tau)$, there exists a unique morphism $\widetilde{\varphi}\colon (A_0,\sigma)\to (C,\tau)$ such that $\widetilde{\varphi} w=\varphi$. Let $\widetilde{\varphi}\colon (A_0,\sigma)\to (C,\tau)$ be the restriction to $A_0$ of $\varphi\colon (X,f)\to (C,\tau)$. Then $$\widetilde{\varphi} w=\widetilde{\varphi} \sigma^{-(n!)}f^{n!}=\tau^{-(n!)}\widetilde{\varphi} f^{n!}=\tau^{-(n!)}\varphi f^{n!}=\tau^{-(n!)}\tau^{n!}\varphi=\varphi.$$ The uniqueness of $\widetilde{\varphi} $ follows from the surjectivity of $w$.

We will now prove that $R$ is the right adjoint of the embedding $\Cal C\hookrightarrow\Cal M$. The counit of the adjunction is the inclusion morphism $\varepsilon\colon (A_0,\sigma)\to(X,f)$. It sufficies to show that, for every object $(X,f)$ of $\Cal M$, every object $(C,\tau)$ of $\Cal C$ and every morphism $\varphi\colon (C,\tau)\to (X,f)$, there exists a unique morphism ${\varphi'}\colon (C,\tau)\to (A_0,\sigma)$ with $\varepsilon\varphi'=\varphi$. Now the morphism $\varphi\colon (C,\tau)\to (X,f)$ maps cycles to cycles, so that $\varphi(C)\subseteq A_0$. Hence the corestriction ${\varphi'}\colon (C,\tau)\to (A_0,\sigma)$ of $\varphi\colon (C,\tau)\to (X,f)$ obtained by restricting the codomain to $A_0$ has the property that $\varepsilon\varphi'=\varphi$. The uniqueness of $\varphi' $ follows from the injectivity of $\varepsilon$.
\end{proof}

\begin{proposition} The category $\Cal F$ is a reflective subcategory of $\Cal M$ which is not coreflecting, that is, the embedding $\Cal F\hookrightarrow\Cal M$ has  a left adjoint, but not a right adjoint. The left adjoint is the precoradical $C$, viewed as a functor of $\Cal M$ to $\Cal F$.\end{proposition}

\begin{proof} To show that $C\colon \Cal M\to\Cal F$ is the left adjoint of the embedding $\Cal F\hookrightarrow\Cal M$, with unit the canonical projection $\pi\colon (X,f)\to (X/\!\!\sim,\overline{f})$, it suffices to prove that for every object $(X,f)$ of $\Cal M$, every object $(F,g)$ of $\Cal F$ and every morphism $\varphi\colon (X,f)\to (F,g)$, there exists a unique morphism $\widetilde{\varphi}\colon (X/\!\!\sim,\overline{f})\to (F,g)$ such that $\widetilde{\varphi} \pi=\varphi$. This is equivalent to proving that $\varphi\colon (X,f)\to (F,g)$ induces a morphism $\widetilde{\varphi}\colon (X/\!\!\sim,\overline{f})\to (F,g)$, that is, that the congruence $\sim$ is contained in the kernel $\sim_\varphi$ of $\varphi$. Now $\sim$, coequalizer of $f^{n}$ and $f^{n+1}$, is the congruence generated by the arrows on the cycles of $f$, so that it suffices to prove that $\varphi$ coequalizes $f^{n}$ and $f^{n+1}$, i.e., that $\varphi f^{n}=\varphi f^{n+1}$. But $\varphi f^{n}=g^{n}\varphi=g^{n+1}\varphi=\varphi f^{n+1}$, as desired.

In order to prove that the embedding $\Cal F\hookrightarrow\Cal M$ does not have a right adjoint, it suffices to show that there exist objects $(X,f)$ in $\Cal M$ for which $\Hom_{\Cal M}(\Cal F,(X,f))=\emptyset$. To this end, it suffices to take as $(X,f)$ any object with $f$ fixed-point-free, for instance a cycle permuting all elements of $X$ ($|X|> 1$). Then every cycle of length zero of any $F$ in $\Cal F$ (that is, any fixed point, which exists in all $F\in \Cal F$), should be mapped to a fixed point of $X$, which does not exist.\end{proof}

\section{A pretorsion theory}\label{QQQ}

The second half of this paper, from Section~\ref{pre} on, has been deeply influenced by the article~\cite{AC}. In that paper, the category $\pre$ of preordered sets $(A,\rho)$ is studied. Here $A$ is any non-empty set, and $\rho$ is a preorder on $A$, that is, a reflexive and transitive relation on $A$. The morphisms $g\colon (A,\rho)\to (A',\rho')$ in the category $\pre$ are the mappings $g$ of $A$ into $A'$ such that $a\rho b$ implies $g(a)\rho' g(b)$ for all $a,b\in A$. In $\pre$, there is a pretorsion theory $(\Equiv,\ParOrd)$, where  $\Equiv$ is the class of all objects $(A,\rho)$ of $\pre$ with $\rho$ an equivalence relation on $A$ and $\ParOrd$ is the class of all objects $(A,\rho)$ of $\pre$ with $\rho$ a partial order on $A$. 

Torsion  theories in general categories are studied in the papers \cite{GJ}, \cite{GJM} and~\cite{JT}. Since these articles are rather technical while our setting is extremely simple, the first author and Carmelo Finocchiaro introduced in \cite{AC}  a much easier notion of pretorsion theory.  We are very grateful to Marino Gran, Marco Grandis and Sandra Mantovani for some suggestions.

Our setting is the following.
  Fix an arbitrary category $\Cal C$ and two non-empty classes $\Cal T,\Cal F$ of objects of $\Cal C$, both closed under isomorphism. Set $\Cal Z:=\Cal T\cap\Cal F$. For every pair $A,B$ of objects of $\Cal C$, $\Triv_{\Cal Z}(A, B)$ indicates the set of  all morphisms in $\Cal C$ that factors through an object of $\Cal Z$. The morphisms in $\Triv_{\Cal Z}(A, B)$ are called $\Cal Z$-trivial.

Let $f\colon A\to A'$ be a morphism in $\Cal C$. A \emph{$\Cal Z$-prekernel} of $f$ is a morphism $k\colon X\to A$ in $\Cal C $ such that: 
\begin{enumerate}
	\item $fk$ is a $\Cal Z$-trivial morphism.
	\item Whenever $ \ell \colon Y\to A$ is a morphism in $\Cal C$ and $f \ell$ is $\Cal Z$-trivial, then there exists a unique morphism $ \ell'\colon Y\to X$ in $\Cal C$ such that $ \ell=k\ell'$. 
\end{enumerate}

The $\Cal Z$-prekernel of $f\colon A\to A'$ turns out to be a subobject of $A$ unique up to isomorphism, when it exists.

Dually, for the \emph{$\Cal Z$-precokernel} of $f$, which  is a morphism $A'\to X$. If
$f\colon A\to B,\ g\colon B\to C$ are morphisms in $\Cal C$, we say that $$\xymatrix{
	A \ar[r]^f &  B \ar[r]^g &  C}$$ is a \emph{short $\Cal Z$-preexact sequence} in $\Cal C$ if $f$ is a $\Cal Z$-prekernel of $g$ and $g$ is a $\Cal Z$-precokernel of $f$.
	
We say that a pair $(\Cal T,\Cal F)$ is a {\em pretorsion theory  for $\Cal C$}, where $\Cal T,\Cal F$ are classes of objects of $\Cal C$ closed under isomorphism and $\Cal Z:=\Cal T\cap\Cal F$, if it satisfies the following two conditions:
  
  (1) For every object $B$ of $\Cal C$ there is a short $\Cal Z$-preexact sequence \begin{center}$\xymatrix{
	A \ar[r]^f &  B \ar[r]^g &  C}$ \end{center}with $A\in\Cal T$ and $C\in\Cal F$.
	
	(2) $\Hom_{\Cal C}(T,F)=\Triv_{\Cal Z}(T, F)$  for every pair of objects $T\in\Cal T$, $F\in\Cal F$.
	
	\bigskip
	
	There is a clear overlapping between this definition of pretorsion theory and the theory developed in \cite{GJM}. In \cite{AC}, a case concerning the category $\Cal C=\pre$ of all non-empty preordered set was studied. The pretorsion theory was the pair $(\Equiv,\ParOrd)$ of sets endowed with an equivalence relation and sets endowed with a partial order, respectively.
	In this paper, we have seen that $(\Cal C,\Cal F)$ is a pretorsion theory in $\Cal M$ ((a) and (b) in Theorem~\ref{111}).
	
	\bigskip
		
There is a functor $U\colon \Cal M\to\pre$, which is a canonical embedding. It associates to any object $(X,f)$ of $\Cal M$ the preordered set $(X,\rho_f)$, where $\rho_f$ is defined, for every $x,y\in X$, setting $x\rho_f y$ if $x=f^t(y)$ for some integer $t\ge0$. Any morphism $g\colon (X,f)\to (X',f')$ in $\Cal M$, is a morphism $g\colon (X,\rho_f)\to (X',\rho_{f'})$ in $\pre$, because if $x,y\in X$ and $x\rho_f y$, then $x=f^t(y)$ for some integer $t\ge0$. The commutativity of diagram~(\ref{homo}) yields that $g(x)=gf^t(y)={f'}^tg(y)$, so $g(x)\rho_{f'} g(y)$. Thus we have a functor $U\colon \Cal M\to\pre$, which is clearly faithful. Hence $\Cal M$ is isomorphic to a subcategory of $\pre$. 

\begin{proposition} The following equalities hold:
\item{\rm (a)} $U(\Cal C) =U(\Ob(\Cal M))\cap\Equiv$.
\item{\rm (b)} $U(\Cal F) =U(\Ob(\Cal M))\cap\ParOrd$.\end{proposition}

\begin{proof} The proof of the inclusion $U(\Cal C) \subseteq U(\Ob(\Cal M))\cap\Equiv$ in (a) is easy. Conversely, assume $(X,f)\in\Ob(\Cal M)$ and $(X,\rho_f)\in\Equiv$. For every $x\in X$, we have that $f(x)\rho_f x$. As $\rho_f$ is an equivalence relation on $X$, we get that $x\rho_ff(x)$. Thus there exists $t\ge 0$ such that $x=f^{t+1}(x)$. It follows that $x\in f(X)$. Therefore the mapping $f$ is onto, i.e., a bijection, so $(X,f)\in\Cal C$.

For (b), suppose $(X,f)\in\Cal F$, i.e., that $G^d_f$ is a forest. In order to prove that $\rho_f$ is a partial order on $X$, assume $x,y\in X$, $x\rho_fy$ and $y\rho_f x$. Then there are $t,u\ge 0$ such that $x=f^t(y)$ and $y=f^u(x)$. Then $x=f^{t+u}(x)$. Since $G^d_f$ is a forest, i.e., has no cycle of length $>1$, it follows that $t+u=0$, so $t=u=0$, and $x=y$. For the opposite inclusion, assume $(X,f)\in\Ob(\Cal M)$ and $(X,\rho_f)\in\ParOrd$. For every $x\in X$, we have that $f^{n+1}(x)\rho_f f^n(x)$. Now $G^d_f$ is a forest on a cycle, and  $f^n(x)$ is on the cycle, so that there exists $v>0$ such that $f^n(x)=f^{n+v}(x)$. Hence $f^n(x)=f^{n+v}(x)\rho_f f^{n+1}(x)$. Therefore $f^n(x)\rho_f f^{n+1}(x)$ and $f^{n+1}(x)\rho_f f^n(x)$, which imply $f^{n+1}(x)= f^n(x)$ because $\rho_f$ is a partial order. Thus $f^n=f^{n+1}$, so $(X,f)\in\Cal F$.\end{proof}

The functor $U$ is not full. For example, consider the permutations $f=(1\ 2\ 3)$ and $g=(1\ 2)$ of $X=\{1,2,3\}$. Then $(X,f)$ is an object of $\Cal M$, but $g$ is not an endomorphism of $(X,f)$ in $\Cal M$, because $fg\ne gf$. Nevertheless $\rho_f$ is the trivial preorder on $X$, in which any two elements of $X$ are in the relation $\rho_f$, so that $g$ is an endomorphism of $(X,\rho_f)$ in the category $\pre$.

\section{The stable category $\underline{\Cal M}$}

Following \cite{AC}, it is now natural to introduce a congruence $R$ on the category $\Cal M$ that identifies all trivial morphisms between two objects of $\Cal M$. Then it is possible to construct the quotient category $
\underline{\Cal M}:=\Cal M/R$, and call it the {\em stable category} of $\Cal M$. But we will see that such a category $\underline{\Cal M}$ and the category $\Cal M$ have a terminal object, but don't have initial objects. So our prekernels and precokernels in $\Cal M$ do not correspond to kernels and cokernels in the stable category  $\underline{\Cal M}$. 

For every pair of objects $(X,f),(X',f')$ of $\Cal M$, there is an equivalence relation $R_{X,X'}$ on the set $\Hom_{\Cal M}((X,f),(X',f'))$ defined, for every $g,h\colon$\linebreak $(X,f)\to(X',f')$, by $g\,R_{X,X'}\,h$ if, for every connected component $C$ of $G^u_f$, either

(1) $g(x)=h(x)$ for every $x\in C$, or

(2) $|g(C)|=1,\ |h(C)|=1,\ f'(g(C))=g(C)$ and $f'(h(C))=h(C)$.

\noindent Here we view the connected components $C$ as subsets of $X$, and $|A|$ denotes the cardinality of the set $A$.

In order to prove that $R_{X,X'}$ is an equivalence relation, reflexivity and symmetry are clear. Transitivity is a little more tricky. Suppose $g,h,\ell\colon$\linebreak $(X,f)\to(X',f')$, $g\,R_{X,X'}\,h$ and $h\,R_{X,X'}\,\ell$. For every connected component $C$ of $G^u_f$, we have four possible cases.

{\em First case: Condition {\rm (1)} holds for both the pairs $(g,h)$ and $(h,\ell)$.} That is, $g(x)=h(x)$ and $h(x)=\ell(x)$ for every $x\in C$. Then $g(x)=\ell(x)$ for every $x\in C$, and we have Condition {\rm (1)} for the pair $(g,\ell)$ as well.

{\em Second case: Condition {\rm (1)} holds for the pair $(g,h)$ and Condition {\rm (2)} holds for the pair $(h,\ell)$.} That is, $g(x)=h(x)$ for every $x\in C$, $|h(C)|=1,\ |\ell(C)|=1,\ f'(h(C))=h(C)$ and $f'(\ell(C))=\ell(C)$. Then $|g(C)|=1$ and $ f'(g(C))=g(C)$, so that Condition {\rm (2)} holds for the pair $(g,\ell)$ as well.

{\em Third case: Condition {\rm (2)} holds for the pair $(g,h)$ and Condition {\rm (1)} holds for the pair $(h,\ell)$.} This is similar to the second case.

{\em Fourth case: Condition {\rm (2)} holds for both the pairs $(g,h)$ and $(h,\ell)$.} Then, trivially, Condition {\rm (2)} holds for the pair $(g,\ell)$ as well.

Now we want to prove that the assignment $(X,X')\mapsto R_{X,X'}$ is a congruence on the category $\Cal M$ in the sense of \cite[page~51]{ML}. Let $g,h\colon (X,f)\to(X',f')$ be morphisms with $g\,R_{X,X'}\,h$, and $$m\colon (Y,g)\to(X,f),\ \ \ell\colon (X',f')\to(X'',f'')$$ be arbitrary morphisms in $\Cal M$. We must show that $gm\,R_{Y,X'}\,hm$ and $\ell g\,R_{X,X''}\,\ell h$.  This is also straightforward, as several of the previous verifications, and we omit it here. 

Therefore it is possible to construct the stable category $\underline{\Cal M}:={\Cal M}/R$ like in \cite{AC}. The difference with \cite{AC} is that now the stable category $\underline{\Cal M}$ does not have a zero object, as we will show in the next paragraph, so that it is not possible to construct kernels and cokernels in $\underline{\Cal M}$ in the classical sense, hence prekernels and precokernels in $\Cal M$ do not correspond to kernels and cokernels in $\underline{\Cal M}$.

The categories $\Cal M$ and $\underline{\Cal M}$ have terminal objects (the singleton $(\{1\}, (1))$ with the identity morphism). To see that $\Cal M$ and $\underline{\Cal M}$  don't have zero objects, notice that there are objects of $\Cal M$, for instance the cycle $(\{1,2,3\},(1\ 2\ 3))$, for which $$\Hom_{\Cal M}((\{1\}, (1)),(\{1,2,3\},(1\ 2\ 3)))$$ is the empty set. Hence $\Hom_{\underline{\Cal M}}((\{1\}, (1)),(\{1,2,3\},(1\ 2\ 3)))$ is the emptyset, which is not possible in a pointed category.


\begin{thebibliography}{99}

\bibitem{AC} A. Facchini and C. A. Finocchiaro, {\em Pretorsion theories, stable category and preordered sets}, to appear, 2019.

\bibitem{GJ} M. Grandis and G. Janelidze {\em From torsion theories to closure operators and
factorization systems,} to appear, 2019.

\bibitem{GJM} M. Grandis, G. Janelidze and L. M\'arki, {\em Non-pointed exactness, radicals, closure operators,} J. Aust. Math. Soc. {\bf 94} (2013), no. 3, 348--361. 
 
\bibitem{H} J. M. Howie, {\em 
The subsemigroup generated by the idempotents of a full transformation semigroup,}
J. London Math. Soc. {\bf 41} (1966), 707--716. 

\bibitem{JT} G. Janelidze and W. Tholen, {\em Characterization of torsion theories in general categories}, in  ``Categories in algebra, geometry and mathematical physics'', A. Davydov, M. Batanin, M. Johnson, S. Lack and A. Neeman Eds., Contemp. Math. {\bf 431}, Amer. Math. Soc., Providence, RI, 2007, pp. 249--256. 

\bibitem{ML} S. MacLane,  ``Categories for the Working Mathematician'', 2nd edn., Springer-Verlag, 1998.

\bibitem{U} V. M. Usenko, {\em On semidirect products of monoids,}  Ukrain. Mat. Zh. {\bf 34} (2) (1982), 185--189, 268.  English translation in Ukrainian Math. J. {\bf 34} (2) (1982), 151--155. Available in 
https://link.springer.com/content/pdf/10.1007/BF01091519.pdf

\end{thebibliography}
\end{document}